\documentclass[11pt]{amsart}
\usepackage[utf8]{inputenc}
\usepackage[leqno]{amsmath}
\usepackage[a4paper, lmargin=0.10\paperwidth, rmargin=0.10\paperwidth, tmargin=0.10\paperheight, bmargin=0.10\paperheight]{geometry}

\usepackage{amssymb,amscd,amsmath,graphicx,enumerate}
\usepackage{hyperref}
\usepackage{graphicx}
\usepackage{pgf,tikz}
\usepackage{mathtools}

\theoremstyle{definition}
\newtheorem{theorem}{Theorem}[section]
\newtheorem{corollary}{Corollary}[section]
\newtheorem{lemma}{Lemma}[section]
\newtheorem{claim}{Claim}[section]
\newtheorem{proposition}{Proposition}[section]
\newtheorem{example}{Example}[section]
\newtheorem{definition}{Definition}[section]
\newtheorem{remark}{Remark}[section]
\newtheorem{notation}{Notation}[section]
\numberwithin{equation}{section}
\newtheorem{question}{Question}[section]
\newtheorem{con}{Conjecture}[section]

\DeclareMathOperator{\reg}{reg}

\begin{document}

\title{d-sequence edge binomials, and regularity of powers of binomial edge ideals of trees}

\author[Marie Amalore Nambi]{Marie Amalore Nambi}
\address{Department of Mathematics, IIT Hyderabad, Kandi, Sangareddy - 502285}
\email{ma19resch11004@iith.ac.in}

\author[Neeraj Kumar]{Neeraj Kumar}
\address{Department of Mathematics, IIT Hyderabad, Kandi, Sangareddy - 502285}
\email{neeraj@math.iith.ac.in}

\subjclass[2020]{{Primary 13F65, 13D02}; Secondary {05E40}} 
\keywords{Binomial edge ideal, Regularity, Tree, $d$-sequence, Edge binomial, Degree sequence.}
\date{First version September 15, 2022. Revised version, March 1, 2023}

\begin{abstract} 
We provide the necessary and sufficient conditions for the edge-binomials of the tree forming a $d$-sequence in terms of the degree sequence notion of a graph. We study the regularity of powers of the binomial edge ideals of trees generated by $d$-sequence edge binomials.
\end{abstract}

\maketitle

\section{Introduction}

Let $S=k[x_1,\ldots,x_n,y_1,\ldots,y_n]$ be a polynomial ring over a field $k$. Let $G$ be a finite simple graph on $[n]$ vertices. Given an edge $\{i,j\}$ in $G$, associate a polynomial $f_{ij}=x_iy_j-x_jy_i$ in $S$. We call $f_{ij}$ an {\emph{edge binomial}}. Denote by $J_G \subset S$ the binomial edge ideal generated by edge binomials. 

\vspace{1mm}

The binomial edge ideal naturally arises in commutative algebra, algebraic geometry, and statistics. For instance, Herzog et al. noticed that the binomial edge ideal appears in the study of conditional independence ideals suitable to investigate robustness theory in the context of algebraic statistics (cf. \cite{HH}). The notion of binomial edge ideal associated with a finite simple graph was coined by Herzog et al. in \cite{HH} and independently by Ohtani in \cite{O2011}. There is a one-to-one correspondence between the binomial edge ideal of a finite simple graph and set of $2$-minors of $2 \times n$ matrix of indeterminates. Diaconis, Eisenbud, and Sturmfels studied the ideal generated by all the adjacent $2$-minors of a $2 \times n$ generic matrix (cf. \cite{DES}). One can see that the binomial edge ideal of path graphs coincides with the ideal of adjacent minors of a $2 \times n$ matrix of indeterminates. In \cite{DES} the authors also studied corner minors of a $2 \times n$ matrix of indeterminates which coincides with the binomial edge ideal of star graphs.
If $G$ is a complete graph, then $S/J_G$ can be visualized as a Segre variety given by the image of Segre product $\mathbb{P}^1 \times \mathbb{P}^{n-1}$ of projective spaces.

\vspace{1mm}

In \cite{V1995}, Villarreal proved that edge ideals are of linear type if and only if the graph is a tree or has a unique cycle of odd length. The binomial edge ideal of paths (ideal of adjacent minors) is generated by regular sequences (cf. \cite{EHH,R2013}). The binomial edge ideal of star graphs (ideal of corner minors) is generated by $d$-sequence (cf. \cite{JAR}). Since  \emph{ complete intersection ideal} $\implies$ \emph{d-sequence ideal} $\implies$ \emph{ideal of linear type} (cf. \cite{H1980}), the binomial edge ideal of paths and star graphs are of linear type. It is natural to ask for combinatorial characterization of linear type binomial edge ideals. Note that there is a strict inclusion among the family of ideals, that is, \emph{almost complete intersection ideal} $\implies$ \emph{d-sequence ideal} $\implies$ \emph{ideal of linear type} (cf. \cite{H1980, H1982}, \cite[p. 341]{H81}).

\vspace{1mm}

The notion of \emph{d-sequence} was introduced by Huneke in \cite{H1980} (see Definition \ref{Def-d-sequence}). An ideal is said to be generated by $d$-sequence if there is a generating set that forms a $d$-sequence. In an attempt to find linear type binomial edge ideals, Jayanthan et al. in \cite{JAR} characterized graphs whose binomial edge ideals are almost complete intersection ideals, a subfamily of linear type ideals. They proposed the following.

\begin{con}
\cite[Conjecture 4.17]{JAR}\label{conj:LT} If the given graph is a tree or a unicyclic graph, then the binomial edge ideal is of linear type.
\end{con}

In \cite{A22}, Kumar characterized linear type binomial edge ideal for closed graphs, namely $J_G$ is of linear type if and only if $G$ is a $K_4$-free graph. In this article, we classify all trees (in terms of degree sequence) whose edge binomials form a $d$-sequence. Our first main result is the following.

\vspace{1mm}

\begin{theorem} \label{P1.1}
Let $G$ be a tree on vertices $x_1,\ldots,x_n$. Let $S=k[x_1,\ldots,x_n,y_1,\ldots,y_n]$ be a polynomial ring. Then the
edge binomials of $G$ form a $d$-sequence in $S$ if and only if $G$ has degree sequences of the following form:
\begin{itemize}
\item[(i)] $(1,1)$;
\item[(ii)] $(m,2,\dots, 1)$ or $(m,1,\dots, 1)$, where $2 \leq m < n$;
\item[(iii)] $(m,3,2,\dots,1)$ or $(m,3,1,\dots,1)$ such that the vertex of degree $m$ and $3$ are adjacent in $G$ with $3 \leq m < n$. 
\end{itemize}
Moreover, the binomial edge ideal $J_G$ is of linear type.
\end{theorem}

The above theorem is significant for at least two reasons: Firstly, it provides a positive answer to a Conjecture \ref{conj:LT} to some extent; secondly, the $d$-sequence notion plays a vital role in studying the regularity of powers of an ideal. 

\vspace{1mm}

The second part of this article is dedicated to applying $d$-sequence  theory to study the Castelnuovo-Mumford regularity of the binomial edge ideals and their powers. Several authors have studied the regularity of the binomial edge ideal of various classes (see, \cite{VRT2021, JNR, A22, CR20, SK, PS, ZZ} for a partial list). Several researchers have explored the regularity of the product of monomial ideals (for instance, see \cite{APS2019} and the reference of this paper). As part of a technical step, we obtained the regularity of the product of binomial edge ideals of disjoint union of paths and a complete graph in Theorem \ref{Thm3.IJ}. The study of regularity and its combinatorial interpretation for the index of stability is a central research topic in combinatorial commutative algebra. It is well known that the regularity of powers of any homogeneous ideal is a linear function, see \cite{CHT1999, V2000}. It is also known that if an ideal is generated by a $d$-sequence of $t$ forms of the same degree $r$, then for all $s \geq t + 1$, $\reg(I^s)=(s-t-1)r + \reg(I^{t+1}),$ in \cite{CHT1999}. Thus computing the regularity of $I^s$, where $I$ is generated by a $d$-sequence and $s>1$, boils down to computing $\reg(I^{t+1})$. Moreover, if an ideal $I \subset S$ generated by a  $d$-sequence $u_1,\ldots,u_n$, and set $u_0 =0 \in S$, then one has $((u_0,u_1,\ldots,u_{i-1})+I^{s}):u_i = ((u_0,u_1,\ldots,u_{i-1}):u_i) + I^{s-1}$, for $s\geq 1$ and $i = 1,\ldots,n$ (see, \cite[Observation 2.4]{SZ}).  

\vspace{1mm}

A subset $U$ of $V(G)$ is said to be \textit{clique} if $G[U]$ is a complete graph. A vertex $v$ of $G$ is said to be a \textit{free vertex} in $G$ if $v$ is contained in only one maximal clique. A vertex of $G$ is called an \emph{internal vertex} if it is not a free vertex of $G$. Given a tree whose edge binomials form a $d$-sequence, we express the regularity of $J_G^s$ in terms of $s$ and the number of internal vertices of $G$. The second main result is below.

 \vspace{1mm}

\begin{theorem} \label{P1.2}
Let $G$ be a tree on vertices $x_1,\ldots,x_n$. Let $S=k[x_1,\ldots,x_n,y_1,\ldots,y_n]$ be a polynomial ring and let $J_G$ be the binomial edge ideal. Let $i(G)$ denote the number of internal vertices of $G$.
\begin{itemize}
\item[(i)] If $G$ has a degree sequence $(m,2,\dots, 1)$ or $(m,1,\dots, 1)$, then $\reg {S}/{J_{G}^s} = 2s + i(G)-1$, for all $s\geq 1$, and $m \geq 2$.
\item[(ii)] If $G$ has a degree sequence $(m,3,2,\dots,1)$ or $(m,3,1,\dots,1)$ such that the vertex of degree $m$ and $3$ are adjacent in $G$, then  $\reg {S}/{J_{G}^s} = 2s + i(G)-1$, for all $s\geq 1$ and $m \geq 3$.
\end{itemize}
\end{theorem}

Note that the stability index is one for each case in the above theorem. In the case of binomial edge ideals,  a few results are known in the literature about the  $\reg J_{G}^s$ for $s > 1$. To our knowledge, some results known in this direction are the following. In \cite{JAR20}, the authors have computed the regularity of powers of binomial edge ideals of cycles and star graphs and obtained the bound for trees and unicycle whose $J_G$ is an almost complete intersection ideal. In \cite{SZ}, the authors obtained the precise bound for the regularity of $J_G^s$ for an almost complete intersection ideal. Ene et al. studied the regularity of $J_G^s$ for closed graphs, in \cite{VRT2021}. We expect that the techniques developed in this article will stimulate further research interest among combinatorial commutative algebraists to study the regularity of powers of the ideal.

\vspace{2mm}

This paper is organized as follows. The following section recalls definitions, notations, and known results. Section 3 characterizes trees, where edge binomials form a $d$-sequence. In Section 4, we compute the regularity of the product of binomial edge ideals of a disjoint union of paths and a complete graph. Section 5 studies the regularity of powers of binomial edge ideals of trees whose edge binomials form a $d$-sequence.

 \vspace{2mm}

 \noindent {\bf Acknowledgement.} 
 The first author is financially supported by the University Grant Commission, India. The second author is partially supported by the Mathematical Research Impact Centric Support (MATRICS) grant from Science and Engineering Research Board (SERB), India. 
 
\section{Preliminaries}

\subsection*{Basic notions from graph theory.}  Let $G$ be a simple graph on the vertex set $V(G) = [n]$ and edge set $E(G)$. The \textit{degree} of a vertex $v \in V(G)$, denoted by $deg_{G}(v)$, is the number of edges incident to $v$. A \emph{degree sequence} of a finite simple graph is a non-increasing sequence of its vertex degrees. The complete graph on $[n]$ is denoted by $K_n$.  Let $C_n$ (and $P_n$) denotes the cycle (path) on $[n]$ vertices respectively. The \textit{length} of a path or cycle is its number of edges. A graph with precisely one cycle as a subgraph is called a \textit{unicyclic graph}. If a connected graph does not contain a cycle, it is called a \textit{tree}. A \textit{star graph} on $m+1$ vertices, denoted by $K_{1,m}$, is the graph with vertex set $V(K_{1,m})=\{v\} \sqcup \{u_1,\ldots,u_m\}$ and edge set $E(K_{1,m})= \{\{v,u_i\} \mid 1\leq i\leq m\}$, and we call the vertex $v$ the center of $K_{1,m}$. A vertex $v$ is said to be a \textit{pendant vertex} if $\deg_{G}(v) = 1$. The \textit{distance} between two vertices in a connected graph is the shortest length of a path between two vertices.

\vspace{1mm}

\begin{definition}\cite[Definition 1.1]{H1982} \label{Def-d-sequence} Let $S$ be a commutative ring. Set $a_0 = 0$. A sequence of elements $a_1,\ldots,a_m$ in $S$ is said to be a $d$-sequence if it satisfies the two conditions: $a_1,\ldots,a_m$ is a minimal system of generators of the ideal $I = (a_1,\ldots,a_m)$; and $( ( a_0,\ldots,a_i ) : a_{i+1}a_j) = (( a_0,\ldots,a_i ) : a_j)$ for all $0 \leq i \leq m-1$, and $j \geq i+1$. 
\end{definition}

\vspace{1mm}

\begin{remark} 
 
There are graphs for which edge binomials do not form a $d$-sequence, but some other minimal generating sets of the binomial edge ideal $J_G$ form a $d$-sequence in $S$. For instance, let $G$ be a unicyclic graph obtained by attaching a path to each vertex of $C_3$. Then $J_G$ is generated by $d$-sequence \cite[Theorem 4.4]{JAR}. However, the edge binomials of $G$ do not form a $d$-sequence \cite[Remark 4.12]{JAR}.
\end{remark}

\begin{notation} Let $G$ be a simple graph on $[n]$. For an edge $e'$ in $G$, $G\setminus e'$ is the graph on the vertex set $V(G)$ and edge set $E(G)\setminus {e'}$. An edge $e'$ is called a \emph{bridge} if $c(G) < c(G \setminus e')$, where $c(G)$ is the number of components of $G$. For a vertex $v$, 
$$N_{G}(v) = \{u \in V(G) \mid \{u,v\} \in E(G)\}$$ 
denotes the \emph{neighborhood} of $v$ in $G$. Let $e =\{i,j\} \notin E(G)$ be an edge in $G \cup \{e\}$. Then $G_e$ (cf. \cite[Definition 3.1]{MS}) is the graph on vertex set $V(G)$ and edge set $$E(G_e) = E(G) \cup \left\{ \{k,l\} : k,l \in N_{G}(i) \textnormal{ or } k,l \in N_{G}(j) \right\}.$$
\end{notation}

We recall some results from the literature relevant to this article. Due to Mohammadi and Sharifan, the following results describe the colon ideal operation on the binomial edge ideal.

\begin{remark}\label{Rem1.MCI} 
\cite[Theorem 3.4]{MS} Let $G$ be a simple graph. Let $e=\{i,j\} \notin E(G)$ be a bridge in $G \cup \{e\}$. Then $J_G: f_e = J_{G_e}$. 
\end{remark}

\begin{remark}\label{Rem1.CI}
Note that for a simple graph $G$, if  $e=\{i,j\} \notin E(G) $ with $1 \leq \deg_{G\cup e}(i), \deg_{G \cup e}(j) \leq 2$ then $G_e = G$, since $\mid N_{G}(i)\mid \leq 1$ and  $\mid N_{G}(j)\mid \leq 1$.
\end{remark}

\begin{lemma} \label{Lemma1.PD}
\cite[Observation 2.4]{SZ}
Suppose that $u_1,\ldots,u_n$ form a $d$-sequence in $S$ and $I = ( u_1,\ldots,u_n )$ is the ideal in $S$. Set $u_0 = 0$. Then, one has 
\[((u_0,u_1,\ldots,u_{i-1})+I^{s}):u_i = ((u_0,u_1,\ldots,u_{i-1}):u_i) + I^{s-1},\]
for $s\geq 1$ and $i = 1,\ldots,n$. 
\end{lemma}

\begin{definition}
Let $S$ be a standard graded polynomial ring over a field $k$. Let $M$ be a finitely generated graded $S$-module. Let $\mathbf{F}_{\bullet}$ be a minimal graded $S$-free resolution of $M$:
$$\mathbf{F}_{\bullet}:  F_{n} \stackrel{\phi_{n}} \longrightarrow  F_{n-1} \longrightarrow  \cdots \longrightarrow  F_1 \stackrel{\phi_{1}} \longrightarrow F_0\longrightarrow  0.$$
Here, $F_i = \oplus_{j}S(-j)^{\beta_{i,j}}$, where $S(-j)$ denotes the graded free module of rank $1$ obtained by shifting the degrees in $S$ by $j$, and $\beta_{i,j}$ denotes the $(i,j)$-th graded Betti number of $M$ over $S$. The \emph{Castelnuovo-Mumford regularity} or simply  \emph{regularity} of $M$ over $S$, denoted by $\reg_S M$, is defined as
$$\reg_SM \coloneqq \max \{j-i \mid \beta_{i,j} \neq 0\}. $$
\end{definition}
For convenience, we shall use $\reg M$ instead of $\reg_SM$. We state the following regularity lemma \cite[Corollary 20.19]{Eisenbud}.

\begin{lemma}[Regularity lemma] \label{Lemma1.Reg}
Let $ 0 \rightarrow M \rightarrow N \rightarrow P   \rightarrow  0$ be a short exact sequence of finitely generated graded $S$-modules. Then the following holds.
\begin{enumerate}[(a)]
    \item $\reg{N} \leq \max \{\reg M, \reg P\}$. The equality holds if $\reg M \neq \reg P +1$.
    \item $\reg{M} \leq \max \{\reg N, \reg P+1\}$. The equality holds if $\reg N \neq \reg P$. 
    \item $\reg{P} \leq \max \{\reg M -1, \reg N\}$. The equality holds if $\reg M \neq \reg N$. 
\end{enumerate}
\end{lemma}

\begin{remark}\label{Rem1.GUG} 
Let $G$ be a new graph obtained by gluing finitely many graphs at free vertices. Let $G=H_1 \cup\cdots\cup H_k$ be a graph satisfying the properties:
\begin{enumerate}
  \item[(a)] For $i \neq j, $ if $H_i \cap H_j \neq  \emptyset $, then $H_i \cap H_j = \{v_{ij}\}$, for some vertex $v_{ij}$ which is free vertex in $H_i$ as well as $H_j$;
  \item[(b)] For distinct $i,j,k, H_i \cap H_j \cap H_k =  \emptyset. $
\end{enumerate}
Then $\reg S/J_G  = \sum_{i=1}^{k} \reg S/J_{H_{i}}$ \cite[Corollary 3.2]{JNR}.
\end{remark}

\begin{remark}\label{rem1.Reg}
Let $G$ be a finite simple graph and $J_G$ be its binomial edge ideal in $S$.
\begin{enumerate}[(a)]
    \item \cite[Theorem 2.1]{SK} If $G=K_n$, then $\reg (S/J_{G}) = 1$, for any $n \geq 3$.
    \item \cite[Theorem 27]{ZZ} \cite[Corollary 3.4]{JNR} Let $i(G)$ be the number of internal vertices of  $G$. If $G$ has a degree sequence either $(m,2,\ldots,1)$ or $(m,1,\ldots,1)$, then $\reg S/J_{G}= i(G)+1$.
\end{enumerate}
\end{remark}

A graph $G$ is called a \emph{block graph} if every block of $G$ is a complete graph. A \emph{flower graph} (cf. \cite[Definition 3.1]{CR20}) $F_{h,k}(v)$ is a connected block graph constructed by joining $h$ copies of the cycle graph $C_3$ and $k$ copies of the bipartite graph $K_{1,3}$ with a common vertex $v$, where $v$ is one of the free vertices of $C_3$ and of $K_{1,3}$, and $cdeg(v) \geq 3$. If $G$ has no ﬂower graphs as induced subgraphs, then $G$ is called flower-free.

\begin{remark} \label{Rem.FF}
    \cite[Corollary 3.2]{CR20} Let $G$ be a connected block graph that does not have an isolated vertex. If $G$ is a flower-free graph, then $\reg S/J_G = i(G) + 1$, where $i(G)$ is the number of internal vertices of $G$.
\end{remark}

\begin{remark} \label{Rem.CIpath}
    \cite[Theorem 1]{R2013} Let $G$ be a graph. Then $J_G$ is a complete intersection if and only if each component of $G$ is a path.
\end{remark}

\section{\textit{d}-sequence edge binomials}

\vspace{2mm}

This section characterizes trees for which edge binomials form a $d$-sequence. 

\vspace{2mm}

The following lemma would describe a characterization for a tree if the corresponding edge binomials of the tree were to form a $d$-sequence. This characterization will play an important role in exploring the colon ideal operation of edge binomials in the form of graph deformation.

\begin{lemma}[Technical Lemma]\label{Lemma2.TL}
Let $G$ be a tree on $[n+1]$. Assume that $a_1,\ldots,a_{n}$ are  $d$-sequence edge binomials of $G$, where $a_k$ corresponds to an edge $\{\alpha_{a_{k}},\beta_{a_{k}}\} \in E(G)$. Let $J_{i}$ denote the ideal generated by $( a_1,\ldots,a_i )$. Let $H_i$ denote the graph associated with $a_1,\ldots,a_i$. 
\begin{enumerate}[(a)]
    \item If $J_i:a_{i+1} = J_i$ for all $i \geq 1$, then edge binomials $a_1,\ldots,a_{n}$ correspond to a path.
    \item If there exists a smallest integer $i$ such that $J_i:a_{i+1} \neq J_i$, then $\{\alpha_{a_{i+1}},\beta_{a_{i+1}}\} \cap \{\alpha_{a_{j}},\beta_{a_{j}}\} \neq \emptyset$, for all $j>i+1$. In particular, for all $j>i+1$, one has $\{\alpha_{a_{i+1}}\} \cap \{\alpha_{a_{j}},\beta_{a_{j}}\} \neq \emptyset$ or $\{\beta_{a_{i+1}}\} \cap \{\alpha_{a_{j}},\beta_{a_{j}}\} \neq \emptyset$. (Intersection property).
\end{enumerate}
\end{lemma}
\begin{proof}
(a) The proof follows from Remark \ref{Rem.CIpath}.

\vspace{2mm}

(b) If $i = n-1$, then proof follows immediately. 
Assume that $i < n-1$. Since $J_i:a_{i+1} \neq J_i$, implies that there exists $f_{kl} \in J_i:a_{i+1}$, where $\{k,l\} \in N_{H_{i}}(\alpha_{a_{i+1}})$ or $\{k,l\} \in N_{H_{i}}(\beta_{a_{i+1}})$, by Remark \ref{Rem1.MCI}.
If $\{k,l\} \in N_{H_{i}}(\alpha_{a_{i+1}})$ then clearly $\{k,l\} \notin N_{H_{i}}(\beta_{a_{i+1}})$,
otherwise $G$ has a cycle.  So one can further assume that $\{k,l\} \in N_{H_{i}}(\alpha_{a_{i+1}})$.
Suppose there exists an edge binomial $a_j$ such that $\{\alpha_{a_{i+1}}\} \cap \{\alpha_{a_{j}},\beta_{a_{j}}\} = \emptyset$, and $\{\beta_{a_{i+1}}\} \cap \{\alpha_{a_{j}},\beta_{a_{j}}\} = \emptyset$.  We have $J_i:a_{i+1}a_j= J_i:a_j$ by hypothesis. This implies that $f_{lk} \in J_{i}:a_j$. Since $f_{kl} \notin J_i$ implies that $\{k,l\} \in N_{H_{i}}(\alpha_{a_{j}})$ or $\{k,l\} \in N_{H_{i}}(\beta_{a_{j}})$. In both cases, $G$ has a cycle; this contradicts the fact that $G$ is a tree. For instance if $\{k,l\} \in N_{H_{i}}(\alpha_{a_{j}})$ the edges $\{\alpha_{a_{i+1}},k\},\{\alpha_{a_{i+1}},l\},\{\alpha_{a_{j}},k\},\{\alpha_{a_{j}},l\}$ forms a cycle. Thus $\{\alpha_{a_{i+1}}\} \cap \{\alpha_{a_{j}},\beta_{a_{j}}\} \neq \emptyset$ or $\{\beta_{a_{i+1}}\} \cap \{\alpha_{a_{j}},\beta_{a_{j}}\} \neq \emptyset$, for all $j > i+1$. Hence the proof is complete.
\end{proof}

\begin{example}
Let $G$ be a tree on $[7]$ with edge set  $$E(G)=\{\{1,2\},\{1,3\},\{1,4\},\{i,j\},\{i_1,j_1\}, \{i_2,j_2\}\}.$$ Let $a_1=f_{12}, a_2=f_{13}, a_3=f_{14}, a_4 = f_{ij}, a_5 = f_{i_1j_1}, a_6 = f_{i_2j_2}$ be a sequence of edge binomials of $G$. One can see that $(J_2:a_3) \neq J_2$, since $f_{23} \in (J_2:a_3)$. Suppose the sequence $a_1,a_2,a_3,a_4,a_5,a_6$ form a $d$-sequence, then the sequence must satisfy the following condition $J_2:a_3a_4 = J_2:a_4$. Since $a_4$ is a bridge, we can conclude that either $\{2,3\} \in N_{H_2}(i)$ or  $\{2,3\} \in N_{H_2}(j)$. Assume that $\{2,3\} \in N_{H_2}(i)$. 

\vspace{2mm}

\textbf{Case (i).} If $i=1$ then it is obvious that $\{1\} \cap \{i,j\} \neq \emptyset$. 

\vspace{2mm}

\textbf{Case (ii).} If $i \neq 1$ then $\{1,2\},\{2,i\},\{i,3\},\{3,1\}$ form a cycle in $G$. That is not possible since $G$ is a tree. Therefore, we may conclude that $\{1\} \cap \{i,j\} \neq \emptyset$.
\end{example}

\begin{notation} \label{defPK} Let
$\mathcal{T}_m$ denote the class of graphs having degree sequence in either of the forms  $(m,2,\ldots,1)$ or $(m,1,\ldots,1)$, where  $m \geq 2$. 

\vspace{2mm}

Equivalently, one can describe $G \in \mathcal{T}_m$ in terms of a vertex set and edge set as below:
$$V(G) = \{k_0,p_{1,1}, \ldots,p_{1,{s(1)+1}}, p_{2,1}, \ldots,p_{2,{s(2)+1}}, \ldots, p_{m,1}, \ldots,p_{m,{s(m)+1}}\}$$
with $s(i) \geq 0$ for all $1 \leq i \leq m$, and edge set  
$$E(G) = \{\{k_0,p_{i,{1}} \mid i = 1,\ldots,m\} \cup \bigcup_{i=1}^{m} \{p_{i,{j}},p_{i,{j+1}} \mid j=1,\ldots,s(i)\}\}.$$

Let $\mathcal{P}$ be an induced subgraph of $G$ on $V(G)\setminus k_0$. We call $\mathcal{P}$ as paths of $G$ and the vertex $k_0$ as the center of $G$. Let $\ell$ denote the number of vertices of degree $2$ in $G$. Notice that $i(G) = \sum_{i=1}^{m}s_{(i)}+1$, for all $m \geq 2$.

\end{notation}

The degree sequence presentation helps us to visualize the graph and the second description is helpful for proving results from the notation point of view. For illustration purpose, a graph $G_1 \in \mathcal{T}_5 $ having degree sequence $(5,2,2,2,1,1,1,1,1)$ is shown in Figure \ref{fig:Tm}.

\begin{figure}[ht]
    \centering
\tikzset{every picture/.style={line width=0.75pt}} 

\begin{tikzpicture}[x=0.65pt,y=0.65pt,yscale=-0.8,xscale=0.8]

\draw    (489,181) -- (452,213) ;
\draw    (489,181) -- (539,197) ;
\draw    (451,264) -- (409,296) ;
\draw    (425,137) -- (415,183) ;
\draw    (452,213) -- (415,183) ; 
\draw    (452,213) -- (412,246) ;
\draw    (452,213) -- (451,264) ;
\draw    (491,247) -- (452,213) ;

\draw (458,206) node [anchor=north west][inner sep=0.75pt]  [font=\footnotesize]  {$k_{0}$};
\draw (418,170) node [anchor=north west][inner sep=0.75pt]  [font=\footnotesize]  {$p_{1,{1}}$};
\draw (410,115) node [anchor=north west][inner sep=0.75pt]  [font=\footnotesize]  {$p_{1,{2}}$};
\draw (488,163) node [anchor=north west][inner sep=0.75pt]  [font=\footnotesize]  {$p_{2,{1}}$};

\draw (544,191) node [anchor=north west][inner sep=0.75pt]  [font=\footnotesize]  {$p_{2,{2}}$};
\draw (495,237) node [anchor=north west][inner sep=0.75pt]  [font=\footnotesize]  {$p_{3,{1}}$};
\draw (451,264) node [anchor=north west][inner sep=0.75pt]  [font=\footnotesize]  {$p_{4,{1}}$};
\draw (396,299) node [anchor=north west][inner sep=0.75pt]  [font=\footnotesize]  {$p_{4,{2}}$};
\draw (395,248) node [anchor=north west][inner sep=0.75pt]  [font=\footnotesize]  {$p_{5,{1}}$};

\filldraw[black] (409,296) circle (2pt) ;
\filldraw[black] (425,137) circle (2pt) ;
\filldraw[black] (489,181) circle (2pt) ;
\filldraw[black] (539,197)  circle (2pt) ;
\filldraw[black] (451,264) circle (2pt) ;
\filldraw[black] (415,183) circle (2pt) ;
\filldraw[black] (412,246) circle (2pt) ;
\filldraw[black] (491,247) circle (2pt) ;
\filldraw[black] (452,213) circle (2pt) ;

\end{tikzpicture}
\caption{The tree $G_1$ with degree sequence $(5,2,2,2,1,1,1,1,1)$.}
 \label{fig:Tm}
\end{figure}
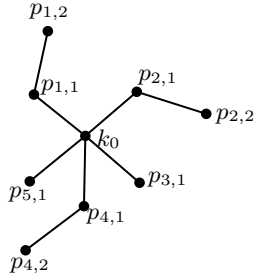

\begin{theorem}\label{Thm2.PtK1m}
Let $G$ be a $\mathcal{T}_{m}$ graph on $[n]$. The edge binomials of $G$ form a $d$-sequence. 
\end{theorem}
\begin{proof}
Set $d_0 =  0 \in S$. We claim $G$ is generated by the $d$-sequence $d_1,\ldots,d_{n-1}$, where the first $n-m-1$ elements are edge binomials of $\mathcal{P}$ (edge binomials of paths of $G$ with respect to any order) and
\[d_i = x_{k_{0}}y_{{p_{i-(n-m-1)}},1}- x_{{p_{i-(n-m-1)}},1}y_{k_{0}}, \text{ for } n-m-1 < i \leq n-1.\] 
Then for all $0 \leq i \leq n-m+1$, and for all $j \geq i+1$,
\begin{equation*} 
\begin{split}
    (d_0,d_1,\ldots,d_{i}):d_{i+1}d_{j} & = ((d_0,d_1,\ldots,d_{i}) :d_{i+1}):d_{j} \\ 
    & = (d_0,d_1,\ldots,d_{i}),
\end{split}
\end{equation*}
also
\begin{equation*} 
 (d_0,d_1,\ldots,d_{i}):d_{j} = (d_0,d_1,\ldots,d_{i}),
\end{equation*}
where the equality follows from Remark \ref{Rem1.CI}, since the graphs associated to the binomial edge ideal $(d_0,d_1,\ldots,d_{i},d_{i+1})$ and $(d_0,d_1,\ldots,d_{i},d_{j})$ are disjoint union of paths. Also for all $n-m+1 < i \leq \ n-2$, and for all $j \geq i+1$,
\begin{equation} \label{eq,d1}
\begin{split}
    (d_0,d_1,\ldots,d_{i}) :d_{i+1}d_{j}  & = ((d_0,d_1,\ldots,d_{i}) :d_{i+1}):d_{j} \\
     & = (d_0,d_1,\ldots,d_{i})+( f_{kl} \mid \{k,l\} \in N(k_{0})),
\end{split}
\end{equation}
and moreover,
\begin{equation} \label{eq,d2}
    (d_0,d_1,\ldots,d_{i}):d_{j} = (d_0,d_1,\ldots,d_{i})+( f_{kl} \mid \{k,l\} \in N(k_{0})),
\end{equation}
where the equality in (\ref{eq,d1}) and (\ref{eq,d2}) follows from Remark \ref{Rem1.MCI} and Lemma \ref{Lemma2.TL}. Thus, we get $(d_0,d_1,\ldots,d_{i}) :d_{i+1}d_{j} = (d_0,d_1,\ldots,d_{i}):d_{j}$, for all $0 < i \leq \ n-2$, and for all $j \geq i+1$. Therefore, the edge binomials of $G$ form a $d$-sequence. 
\end{proof}

\begin{corollary}
Let $G$ be a $\mathcal{T}_{m}$ graph on $[n]$. Then $J_G$ is of linear type.
\end{corollary}
\begin{proof}
The proof follows from \cite[Theorem 3.1]{H1980}.
\end{proof}

\begin{notation}
 \label{Def:Hm}
Let $\mathcal{H}_m$ be a class of trees having a degree sequence in either of the form $(m,3,2,\ldots, 1)$ or $(m,3,1,\ldots, 1)$ such that the vertex of degree $m$ and $3$ are adjacent, where  $m \geq 3$.

\vspace{2mm}

Equivalently, one can describe $G \in \mathcal{H}_m$ in terms of a vertex set and edge set as below:
$$V(G) = \{k_0,k_1,p_{1,1}, \ldots,p_{1,{s(1)+1}}, p_{2,1}, \ldots,p_{2,{s(2)+1}}, \ldots, p_{{m+1},1}, \ldots,p_{{m+1},{s(m+1)+1}}\}$$
with $s(i) \geq 0$ for all $1 \leq i \leq m+1$, and edge set 
\begin{equation*}
    \begin{split}
        E(G) = &\{\{k_0,p_{i,{1}} \mid i = 1,\ldots,m-1\} \cup \{k_0,k_1\} \cup \{k_1,p_{i,{1}} \mid i = m,m+1\} \\
        &\bigcup_{i=1}^{m+1} \{p_{i,{j}},p_{i,{j+1}} \mid j=1,\ldots,s(i)\}\}.
    \end{split}
\end{equation*}

\end{notation}
For illustration purpose, a graph $G_2 \in \mathcal{H}_4 $ having degree sequence $(4,3,1,1,1,1,1)$ is shown in Figure \ref{fig:Hm}. Notice that $i(G) = \sum_{i=1}^{m+1}s_{(i)}+2$, for all $m \geq 3$.

\begin{figure}[ht]
    \centering

\tikzset{every picture/.style={line width=0.75pt}} 

\begin{tikzpicture}[x=0.75pt,y=0.75pt,yscale=-1,xscale=1]

\draw    (371,533) -- (342.42,510.92) ;
\draw    (371,533) -- (370.42,503.92) ;
\draw    (371,533) -- (398.42,511.92) ;
\draw    (372.25,558.92) -- (394.67,583.83) ;
\draw    (372.25,558.92) -- (348.67,583.83) ;
\draw    (371,533) -- (372.25,558.92) ;

\draw (372,530) node [anchor=north west][inner sep=0.75pt]  [font=\scriptsize]  {$k_{0}$};
\draw (326,497) node [anchor=north west][inner sep=0.75pt]  [font=\scriptsize]  {$p_{1,1}$};
\draw (364,489) node [anchor=north west][inner sep=0.75pt]  [font=\scriptsize]  {$p_{2,1}$};
\draw (396,497) node [anchor=north west][inner sep=0.75pt]  [font=\scriptsize]  {$p_{3,1}$};
\draw (365,562) node [anchor=north west][inner sep=0.75pt]  [font=\scriptsize]  {$k_{1}$};
\draw (337,589) node [anchor=north west][inner sep=0.75pt]  [font=\scriptsize]  {$p_{5,1}$};
\draw (394.67,589) node [anchor=north west][inner sep=0.75pt]  [font=\scriptsize]  {$p_{4,1}$};

\filldraw[black] (371,533) circle (1.5pt) ;
\filldraw[black] (342.42,510.92) circle (1.5pt) ;
\filldraw[black] (370.42,503.92) circle (1.5pt) ;
\filldraw[black] (398.42,511.92) circle (1.5pt) ;
\filldraw[black] (372.25,558.92)  circle (1.5pt) ;
\filldraw[black] (394.67,583.83) circle (1.5pt) ;
\filldraw[black] (348.67,583.83)  circle (1.5pt) ;
\end{tikzpicture}
    \caption{The tree $G_2$ with degree sequence $(4,3,1,1,1,1,1)$.}
    \label{fig:Hm}
\end{figure}
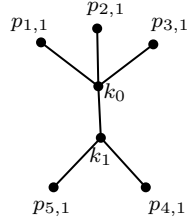

\begin{theorem} \label{Thm2.T2}
Let $G$ be a $\mathcal{H}_{m}$ graph on $[n]$. The edge binomials of $G$ form a $d$-sequence. In particular, $J_G$ is of linear type.
\end{theorem}
\begin{proof}
Consider the following edge binomial sequence $d_1,\ldots,d_{n-1}$, where the first $n-m-1$ elements are edge binomials of $\mathcal{P}$, where $\mathcal{P}$ denotes an induced subgraph of $G$ on $V(G) \setminus k_0$ (independent of order), $$d_i = x_{k_{0}}y_{p_{{i-(n-m-1)},1}} - x_{p_{{i-(n-m-1)},1}}y_{k_{0}},$$ for $n-m-1 < i \leq n-2$, and $$d_{n-1} = x_{k_{0}}y_{k_{1}}-x_{k_{1}}y_{k_{0}}.$$
The rest of the proof is similar to the proof of Theorem \ref{Thm2.PtK1m}.
\end{proof}

We recall a couple of definitions from \cite{SZ}. A graph $G$ is of $T$-type graph if $G$ is obtained by adding an edge between an internal vertex of a path and a pendant vertex of another path. A graph $G$ is said to be a $H$-type graph if $G$ is obtained by adding an edge between two internal vertices of two distinct paths.

\begin{remark} The authors in \cite[Theorem 4.3.]{JAR} have characterized trees whose binomial edge ideals are almost complete intersections. Also, the authors proved that almost complete intersection binomial edge ideals are generated by a $d$-sequence. Shen and Zhu \cite{SZ} classified almost complete intersection binomial edge ideals into two types of graphs called $T$-type and $H$-type. One can see that the $T$-type graph is an induced subgraph of $\mathcal{T}_{m}$ graph, and the $H$-type graphs are an induced subgraph of $\mathcal{H}_{m}$ graph.
\end{remark}

\begin{proposition} \label{Pro:1}
If edge binomials of a tree form a $d$-sequence, then the tree is either a path  $ P_2$ or a graph in  $\mathcal{T}_m \cup \mathcal{H}_m$.
\end{proposition}
\begin{proof}
 Let $G$ be tree on $[n]$ such that edge binomials $d_1,\ldots,d_{n-1}$ of $G$ form a  $d$-sequence. Let $H_k$ be a tree associated with edge binomials $d_1,\ldots,d_{k}$ for $k=1,\ldots,n$. Suppose there exists the smallest integer $i$ such that $J_i:d_{i+1} \neq J_i$, then from Remark \ref{Rem1.MCI} it follows that $H_{i+1}$ has a vertex of degree $3$. Let $k_0$ be a vertex of degree $3$ in $H_{i+1}$. From Remark \ref{Rem.CIpath} it follows that $H_i$ is a disjoint union of paths. From Lemma \ref{Lemma2.TL}(b) it follows that edges associated with edge binomials $d_j$ for all  $i+2 \leq j \leq n-1$, intersect with $k_0$. This implies that $\deg_G(k_0)=n-i+1$. Set $m=n-i+1$. Since $G$ is a tree, $G$ can be obtained by identifying one free vertex of each path with a free vertex of $K_{1,m}$. Now we consider all the possible ways in the following cases:

\vspace{2mm}

\noindent \textbf{Case 1.}  If one free vertex of each path is identified with distinct pendant vertices of $K_{1,m}$, then $G$ is isomorphic to a graph in $\mathcal{T}_{m}$. 

\vspace{2mm}

\noindent \textbf{Case 2.} Let $k_1$ be a pendant vertex of $K_{1,m}$. Suppose the vertex $k_1$ is identified with two free vertices of two distinct paths. One free vertex from each path is identified from the remaining paths with distinct pendant vertices of $K_{1,m}\setminus k_{1}$. In that case,  $G$ is isomorphic to a graph in $\mathcal{H}_{m}$. 

\vspace{2mm}

\noindent \textbf{Case 3.} If three (or more) pendant vertices of three (or more) distinct paths are identified with a pendant vertex of $K_{1,m}$ then it contradicts the choice of the smallest integer $i$.

\vspace{2mm}

\noindent \textbf{Case 4.} Let $k_1$ and $k_2$ be two pendant vertices of $K_{1,m}$. Let $d_p = f_{k_0k_1}$ and $d_q = f_{k_0k_2}$ be edge binomials, where $p,q >i$. Suppose the vertex $k_1$ is identified with two free vertices of two distinct paths and $k_2$ is identified with two free vertices of another two distinct paths, then $J_{p-1}:d_pd_q \neq J_{p-1}:d_q$, where $p <  q$, by Remark \ref{Rem1.MCI}. Therefore, in this case, any sequence of edge binomials of $G$ does not satisfy the $d$-sequence condition.

\vspace{2mm}

If no such $i$ exists, then from Lemma \ref{Lemma2.TL}(a) it follows that $G$ is a path.
\end{proof}

As an immediate consequence, we have the following:
\begin{corollary} \label{Cor2.t1}
Let $G$ be a tree. If $G$ has at least two vertices with degrees greater than or equal to $4$ or $G$ has at least three vertices with degrees greater than or equal to $3$, then any sequence of edge binomials of $G$ does not form a $d$-sequence. 
\end{corollary}
\begin{corollary} \label{Cor2.t2}
Let $G$ be a tree. If there exists $u,v \in V(G)$ such that $\deg({u}), \deg({v}) \geq 3$ and $d(u,v) \geq 2$. Then any sequence of edge binomials of $G$ does not form a $d$-sequence.
\end{corollary}

\subsection*{Conclusion} The proof of Theorem~\ref{P1.1} follows from Theorems \ref{Thm2.PtK1m}, \ref{Thm2.T2} and Proposition \ref{Pro:1}.

\section{Regularity of the product of a disjoint union of paths and a complete graph}
In this section, we state the regularity of the binomial edge ideals of $d$-sequence graphs. The results of this section will be used in Section $5$. 

\vspace{2mm}

\begin{lemma} \label{Thm3.T2}
Let $G$ be a tree. If $G \in \{\mathcal{T}_m,\mathcal{H}_m\}$ graph, then $\reg {S}/{J_{G}} =i(G) + 1$.
\end{lemma}
\begin{proof}
Observe that any $G\in \{\mathcal{T}_m,\mathcal{H}_m\}$ is a flower-free block graphs. Thus the result follows from Remark \ref{Rem.FF}.
\end{proof}

\begin{notation} \label{Def1.g}
Let $n \geq 3$, and $m \geq 1$.
$\mathcal{C}_{n,m}$ be a graph on $[n+m]$ such that 
$$V(\mathcal{C}_{n,m}) = \{v_0,v_1,\ldots,v_{n-1},k_1,\ldots,k_m\},$$ and edge set $$E(\mathcal{C}_{n,m}) = \{\{v_i,v_j \mid 0\leq i<j \leq n-1\}, \{v_0,k_i \mid 1 \leq i \leq m\} \}.$$
\end{notation}

The $\mathcal{C}_{3,2}$ graph is illustrated in Figure \ref{fig:1}.
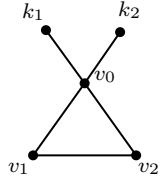
\begin{figure}[ht] \label{fig:1}
    \centering
\tikzset{every picture/.style={line width=0.75pt}} 

\begin{tikzpicture}[x=0.75pt,y=0.75pt,yscale=-1,xscale=1]


\draw   (533.54,536.67) -- (559.25,572.92) -- (507.83,572.92) -- cycle ;
\draw    (514.25,509.92) -- (533.54,536.67) ;
\draw    (551.25,510.92) -- (533.54,536.67) ;


\draw (536.83,528.67) node [anchor=north west][inner sep=0.75pt]  [font=\scriptsize]  {$v_{0}$};
\draw (500.83,494.67) node [anchor=north west][inner sep=0.75pt]  [font=\scriptsize]  {$k_{1}$};
\draw (548.83,493.67) node [anchor=north west][inner sep=0.75pt]  [font=\scriptsize]  {$k_{2}$};
\draw (493.83,575) node [anchor=north west][inner sep=0.75pt]  [font=\scriptsize]  {$v_{1}$};
\draw (558.83,575) node [anchor=north west][inner sep=0.75pt]  [font=\scriptsize]  {$v_{2}$};

\filldraw[black] (533.54,536.67) circle (1.5pt) ;
\filldraw[black] (551.25,510.92) circle (1.5pt) ;
\filldraw[black] (514.25,509.92) circle (1.5pt) ;
\filldraw[black] (507.83,572.92) circle (1.5pt) ;
\filldraw[black] (559.25,572.92)  circle (1.5pt) ;

\end{tikzpicture}
    \caption{The graph $\mathcal{C}_{3,2}$.} 
\end{figure}

\begin{lemma} \label{Lemma3.KnK1m}
Let $G$ be a $\mathcal{C}_{n,m}$ graph. Then $\reg {S}/{J_G} = 2$, for any $m \geq 1$ and $n\geq 3$.
\end{lemma}
\begin{proof}
The proof follows from Remark \ref{Rem.FF}.
\end{proof}

Next, we obtain the regularity of the product of binomial edge ideals of the disjoint union of paths and a complete graph.
\begin{theorem} \label{Thm3.IJ}
Let $H=\{P_1',\ldots ,P_t'\}$ be a disjoint union of paths and $K_{m}$ be a complete graph on $[m]$ such that:
\begin{enumerate}[(a)]
    \item \label{item1} For any $i$, if $K_m \cap P_i' \neq \emptyset$, then $V(K_m) \cap V(P_i') = v_i$, for some $v_i$ which is free vertex in $P_i'$;
    \item \label{item2} $V(K_m) \cap V(P_i') \cap V(P_j')= \emptyset$, for all distinct $i$ and $j$.
\end{enumerate}
Let $n$ be the number of edges in $H$. Then, for any $n \geq 1$ and for any $m \geq 2$, we have 
$$\reg \frac{S}{J_{H}J_{K_{m}}} = 2+n.$$
\end{theorem}
\begin{proof}
    Let $E(H)=\{e_1,\ldots,e_n\}$ be the edges of the disjoint union of paths. First, we claim $(J_{H}J_{K_{m}}) = (J_H) \cap (J_{K_{m}})$. Clearly the inclusion  $(J_{H}J_{K_{m}}) \subseteq (J_H) \cap (J_{K_{m}})$,  holds.

For other side inclusion, let $ x \notin (J_{H}J_{K_{m}})$ and $x \in J_H$. Then $x = \sum_{i=1}^{n} q_if_{e_i}$, where $q_i \in S$. Now suppose $x \in J_{K_{m}}$ and $q_j \notin J_{K_{m}}$, where $j=1,\ldots,n$. Then from \cite[Theorem 3.7]{MS} it follows that there exists a path from $\alpha_j$ to $\beta_j$ in $K_m$, where $\{\alpha_j,\beta_j\}= e_j$. This is not possible by hypothesis (\ref{item1}) and (\ref{item2}). Therefore, $q_j \in J_{K_{m}}$ this implies that $x \in (J_{H}J_{K_{m}})$ which is a contradiction. Similarly, let $ x \notin (J_{H}J_{K_{m}})$ and $x \in J_{K_{m}}$. Then $x = \sum_{k_i \in J_{K_m}} r_ik_i$, where $r_i \in S$. Now, suppose $x \in J_H$ and $r_i \notin J_{H}$. Then from \cite[Theorem 3.7]{MS} it follows that there exists a path from $\alpha_i$ to $\beta_i$ in $H$, where $\{\alpha_i,\beta_i\}$ is an edge associated with an edge binomial $k_i \text{ for any } i$. This is not possible by hypothesis (\ref{item1}) and (\ref{item2}). Therefore, $r_i \in J_H$ this implies that $x \in (J_{H}J_{K_{m}})$ which is a contradiction. Hence, $(J_{H}J_{K_{m}}) \supseteq (J_H) \cap (J_{K_{m}})$, holds.

 Consider the following  short exact sequences:
\begin{equation}\label{eq:1}
    0 \longrightarrow \frac{S}{(J_{H}\cap J_{K_{m}})} \longrightarrow \frac{S}{J_{H}} \oplus \frac{S}{J_{K_{m}}}   \longrightarrow \frac{S}{(J_{H}+J_{K_{m}})}   \longrightarrow  0.
\end{equation}

From Remark \ref{Rem1.GUG} and Remark \ref{rem1.Reg}  it follows that $\reg S/J_H = n$ and $\reg S/J_{K_m} = 1$. From Remark \ref{Rem1.GUG} and Remark \ref{Rem.FF} it follows that $\reg S/(J_H+J_{K_m}) = n+1$, since the graph $H \cup K_m$ is flower-free. Applying Lemma \ref{Lemma1.Reg}(b) to the short exact sequence (\ref{eq:1}) we get 
$$\reg {S}/{(J_{H}\cap J_{K_{m}})}=\reg {S}/{(J_{H}J_{K_{m}})} =n+2.$$
\end{proof}

\section{Regularity of powers of \textit{d}-sequence binomial edge ideal}
In this section, we obtain precise expressions for the regularity of powers of the binomial edge ideals of $d$-sequence graphs. 

\vspace{2mm}

 The authors in \cite{JNR} computed the regularity of binomial edge ideals of $\mathcal{T}_{m}$ graphs. We now calculate the regularity of powers of such binomial edge ideals.
\begin{theorem} \label{Thm4.T1}
Let $G$ be a $\mathcal{T}_{m}$ graph on $[n+1]$. Let  $d_1,\ldots,d_n$ be a sequence of edge binomials of $G$ such that $d_1,\ldots,d_n$ forms a $d$-sequence. Set $d_0 =  0 \in S$. Then, for all $ s \geq 1$, and $m \geq 2$, and $i=0,1,\ldots,n-1$, we have
$$\reg \frac{S}{(d_1,\ldots,d_{i})+ J_{G}^{s}} = 2s + \sum_{j=1}^m{s_{(j)}},$$
where $s_{(j)}$ as defined in Notation \ref{defPK}.
In particular, $\reg {S}/{J_{G}^s} = 2s + \sum_{j=1}^{m}{s_{(j)}}$, for all $s\geq 1$.
\end{theorem}

The edge binomials of $\mathcal{T}_{m}$ graph form a $d$-sequence, and we take $d_1,\ldots,d_n$ to be in the same order as in Theorem \ref{Thm2.PtK1m}. 
To complete the proof of Theorem \ref{Thm4.T1}, we need to prove a  couple of lemmata.

\begin{lemma} \label{Lemma4.T1s=2}
Under the assumption in Theorem \ref{Thm4.T1}, and for any $i=0,1,\ldots,n-1$, we have
$$\reg \frac{S}{(d_1,\ldots,d_i)+J_{G}^2} = 4 + \sum_{j=1}^m{s_{(j)}}.$$
\end{lemma}
\begin{proof}
Consider the following short exact sequence 
\begin{equation} \label{eq:E1}
    0 \longrightarrow \frac{S}{(d_1,\ldots,d_{n-1}):d_n}(-2) \longrightarrow \frac{S}{(d_1,\ldots,d_{n-1})} 
  \longrightarrow \frac{S}{(d_1,\ldots,d_{n})}   \longrightarrow  0. 
\end{equation}

From Remark \ref{Rem1.MCI}, it follows that ideal  $((d_1,\ldots,d_{n-1}):d_n)$ is binomial edge ideal of graph obtained by gluing ${K_{m}}$ and paths at free vertices. Thus from Remark \ref{Rem1.GUG} it follows that $\reg {S}/{((d_1,\ldots,d_{n-1}):d_n)} = 1 + \sum_{j=1}^{m}s_{(j)}$. From Remark \ref{rem1.Reg}(b) it follows that $\reg {S}/{J_G} = 2 + \sum_{j=1}^{m}s_{(j)}$. Applying Lemma \ref{Lemma1.Reg}(a) to the exact sequence (\ref{eq:E1}), yields that 
  $$\reg \frac{S}{(d_1,\ldots,d_{n-1})} \leq 2 + \sum_{j=1}^{m}s_{(j)}.$$
Now, the proof of the lemma is by descending induction on $i$.
For $i=n-1$, Consider the short exact sequence 
\begin{equation} \label{eq:E2}
    0 \longrightarrow \frac{S}{(d_1,\ldots,d_{n-1}):d_{n}^{2}}(-4) \longrightarrow \frac{S}{(d_1,\ldots,d_{n-1})} \longrightarrow \frac{S}{(d_1,\ldots,d_{n-1},d_{n}^{2})}   \longrightarrow  0.
\end{equation}

Since sequence of edge binomials $d_1,\ldots,d_{n}$ form a $d$-sequence, one has   $(d_1,\ldots,d_{n-1}):d_{n}^{2}=(d_1,\ldots,d_{n-1}):d_{n}$. As $\reg \frac{S}{(d_1,\ldots,d_{n-1}):d_{n}^{2}}(-4) = 5 + \sum_{j=1}^{m}s_{(j)} > \reg \frac{S}{(d_1,\ldots,d_{n-1})}$,  it follow from Lemma \ref{Lemma1.Reg}(c) that 
\[
\reg \frac{S}{(d_1,\ldots,d_{n-1})+J_{G}^{2}} = \reg \frac{S}{(d_1,\ldots,d_{n-1})+d_{n}^{2}} = 4 + \sum_{j=1}^{m}s_{(j)}.
\]
Assume that the assertion holds for $j$. Consider the short exact sequence 
\begin{equation} \label{eq:E3}
\begin{split}
    0 \longrightarrow \frac{S}{((d_1,\ldots,d_{j-1})+J_{G}^{2}):d_{j}}(-2) & \longrightarrow \frac{S}{(d_1,\ldots,d_{j-1})+J_{G}^{2}} \\
 & \longrightarrow \frac{S}{(d_1,\ldots,d_{j})+J_{G}^{2}}   \longrightarrow  0. 
\end{split}
\end{equation}

From Lemma \ref{Lemma1.PD}, it follows that  $((d_1,\ldots,d_{j-1})+J_{G}^{2}):d_{j} = ((d_1,\ldots,d_{j-1}):d_{j})+J_{G}$. From Remark \ref{Rem1.MCI}, $((d_1,\ldots,d_{j-1}):d_{j})+J_{G} = J_{G} + ( f_{lk} \mid \{l,k\} \in N_{H_{i}}(\alpha_{d_{i+1}})$ or $\{l,k\} \in N_{H_{i}}(\beta_{d_{i+1}}))$, is the binomial edge ideal of graph obtained by gluing $\mathcal{C}_{n,m-n+1}$ and paths at free vertices, where $n<m$. Thus, from Lemma \ref{Lemma3.KnK1m}, and Remark \ref{Rem1.GUG}, it follows that $ \reg {S}/{(((d_1,\ldots,d_{j-1}):d_{j})+J_{G})} = 2 + \sum_{j=1}^{m}s_{(j)}$. By the induction hypothesis conclude that $\reg {S}/{((d_1,\ldots,d_{j})+J_{G}^{2})}= 4 + \sum_{j=1}^{m}s_{(j)}$. Then by applying Lemma \ref{Lemma1.Reg}(a) to (\ref{eq:E3}), we get 
\[
\frac{S}{(d_1,\ldots,d_{j-1})+J_{G}^{2}}=4 + \sum_{j=1}^{m}s_{(j)},
\]
as desired.
\end{proof}

\begin{lemma} \label{Lemma4.T1p2}
Under the assumption in Theorem \ref{Thm4.T1}, and for all $s>1$, we have 
$$\reg \frac{S}{(d_1,\ldots,d_{n-1})+ d_{n}^{s}} = 2s + \sum_{j=1}^m{s_{(j)}}.$$
\end{lemma}
\begin{proof}
Consider the following short exact sequence of graded modules:
\begin{equation} \label{eq:E4}
    0 \longrightarrow \frac{S}{(d_1,\ldots,d_{n-1}):d_{n}^{s}}(-2s) \longrightarrow \frac{S}{(d_1,\ldots,d_{n-1})} 
  \longrightarrow \frac{S}{(d_1,\ldots,d_{n-1})+d_{n}^{s}}   \longrightarrow  0. 
\end{equation}
Since $d_1,\ldots,d_{n}$ form a $d$-sequence, one has $(d_1,\ldots,d_{n-1}): d_{n}^{s} = (d_1,\ldots,d_{n-1}): d_{n}$. From Lemma \ref{Lemma4.T1s=2}, we know the regularity of modules ${S}/{((d_1,\ldots,d_{n-1}):d_n)}$ and ${S}/{(d_1,\ldots,d_{n-1})}$. Then from Lemma \ref{Lemma1.Reg}(c) it follows that $\reg {S}/{((d_1,\ldots,d_{n-1})+d_{n}^{s})} = 2s + \sum_{j=1}^{m}s_{(j)}$.
\end{proof}

\begin{proof}[Proof of Theorem~\ref{Thm4.T1}] 
The proof is by induction on $s$. For $s=1$, the statement follows from Remark \ref{rem1.Reg}(b).  We may assume that the assertion holds for $s-1$. The proof for the statement $s$ is by descending induction on $i$. For $i=n-1$, the statement holds by Lemma \ref{Lemma4.T1p2}. Next, assume that the assertion holds for $i+1$. Consider the short exact sequence
\begin{equation} \label{eq:T1}
\begin{split}
     0 \longrightarrow \frac{S}{(d_1,\ldots,d_{i})+J_{G}^{s}:d_{i+1}}(-2) & \longrightarrow \frac{S}{(d_1,\ldots,d_{i})+J_{G}^{s}} \\
      & \longrightarrow \frac{S}{(d_1,\ldots,d_{i+1})+J_{G}^{s}}   \longrightarrow  0. 
\end{split}
\end{equation}
By the induction hypothesis on $i$ that $\reg {S}/{(d_1,\ldots,d_{i+1})+J_{G}^{s}} = 2s + \sum_{j=1}^m{s_{(j)}}$. By Lemma \ref{Lemma1.PD} one has $((d_1,\ldots,d_{i})+J_{G}^{s}):d_{i+1} = ((d_1,\ldots,d_{i}):d_{i+1})+J_{G}^{s-1}$. By Remark \ref{Rem1.MCI}, it follows that $((d_1,\ldots,d_{i}):d_{i+1})+J_{G}^{s-1} =(d_1,\ldots,d_{i})+J_{G}^{s-1} + ( f_{lk} \mid \{l,k\} \in N_{H_{i}}(\alpha_{d_{i+1}})$ or $\{l,k\} \in N_{H_{i}}(\beta_{d_{i+1}}))$. Note that $( f_{lk} \mid \{l,k\} \in N_{H_{i}}(\alpha_{d_{i+1}})$ or $\{l,k\} \in N_{H_{i}}(\beta_{d_{i+1}})) = J_{K_{n}}$, for some $n < m$, by Lemma \ref{Lemma2.TL}. 
Set $I= (d_1,\ldots,d_{i})+J_{G}^{s-1}$ and $J=J_{K_{n}}$. Consider the two cases: 

\vspace{2mm}

\textbf{Case 1.}
 If  $((d_1,\ldots,d_{i}):d_{i+1})+J_{G}^{s-1} = I$, then from the induction hypothesis on $s$, it follows that $\reg {S}/{I} = 2(s-1) + \sum_{j=1}^m{s_{(j)}}$. Now, using Lemma \ref{Lemma1.Reg}(a) to the short exact sequence (\ref{eq:T1}) yields $$\reg \frac{S}{(d_1,\ldots,d_{i})+J_{G}^{s}}= 2s + \sum_{j=1}^m{s_{(j)}}.$$
 
\vspace{2mm}

\textbf{Case 2.}
 If $((d_1,\ldots,d_{i}):d_{i+1})+J_{G}^{s-1} =I + J$, then from the following short exact sequence 
 \begin{equation} \label{eq:IJ}
    0 \longrightarrow \frac{S}{I \cap J} \longrightarrow \frac{S}{I} \oplus \frac{S}{J}
  \longrightarrow \frac{S}{I+J}   \longrightarrow  0,
 \end{equation}
one has, 
$$\reg{\frac{S}{I+J}} = \max \{\reg{\frac{S}{I}}, \reg{\frac{S}{J}}, \reg{\frac{S}{I\cap J}}-1\},$$
if $\reg{{S}/{(I\cap J)}} \neq \max \{\reg{{S}/{I}},\reg{{S}/{J}}\}$. 
From the induction hypothesis on $s$, it follows that $\reg {S}/{I} = 2(s-1) + \sum_{j=1}^m{s_{(j)}}$. Since $J$ is a complete graph, from Remark \ref{rem1.Reg}(a) it follows that $\reg {S}/{J} = 1$. 
\begin{claim} \label{Claim1}
$I \cap J = J \cdot( x_{k_{0}},y_{k_{0}},J_{\mathcal{P}} )$, where $\mathcal{P}$ denotes the paths of $G$ and $k_0$ denote the center of $G$.
\end{claim}
For $I \cap J \supseteq J \cdot ( x_{k_{0}},y_{k_{0}},J_{\mathcal{P}} )$ inclusion, it is enough to show that $\{f_{lk}x_{k_{0}},f_{lk}y_{k_{0}}\} \in I$ for all $f_{lk} \in J$, since $J_{\mathcal{P}} \subset I$. For any $f_{lk} \in J$ we have $f_{lk_{0}} = d_{i_1}$ and  $f_{kk_{0}} = d_{i_2}$, for some $ i_1,i_2 < i$. This implies that $\{f_{lk_{0}}, f_{kk_{0}}\} \in I$. Then $x_lf_{kk_{0}} - x_kf_{lk_{0}} = x_{k_{0}}f_{lk}$ and $y_lf_{kk_{0}} - y_kf_{lk_{0}} = y_{k_{0}}f_{lk}$, as desired.
For other side inclusion, for $i < n-1$, let  
 $H_1$ be a graph associated to binomials $d_1,\ldots,d_i$,
 $H_2$ be a graph associated to binomials $d_{i+1},\ldots,d_n$, and
 $H_3$ be a graph such that $E(H_{3})=\{k,l\}$ where $\{k,l\} \in N_{H_{1}}(\alpha_{d_{i+1}})$ or $\{k,l\} \in N_{H_{1}}(\beta_{d_{i+1}})$. 
 Notice that $H_{r_1} \cap H_{r_2} = \emptyset$ where $r_1,r_2 \in \{1,2,3\}$,
 $I=J_{H_{1}} + J_{H_{2}}^{s-1}$ and $J=J_{H_{3}}$. 
 
 \vspace{2mm}
 
Assume that $x \notin  J \cdot (x_{k_{0}},y_{k_{0}},J_{\mathcal{P}})$ and $x \in J$. Thus, $x=\sum_{g_i \in J_{H_{3}}} r_ig_i$, where $r_{i} \in S$. Now, suppose $x \in I$, then $r_{i} \in I$ for all $i$, since $(E(H_{1}) \cup E(H_{2}))\cap E(H_{3}) = \emptyset$. Therefore, $x \in IJ$, but $IJ \subset J\cdot ( x_{k_{0}},y_{k_{0}},J_{\mathcal{P}} )$, which is a contradiction. Thus $x\notin I$. In particular, $x\notin I \cap J$.   Now, assume that $x \notin  J \cdot ( x_{k_{0}},y_{k_{0}},J_{\mathcal{P}} )$ and $x \in I$. Thus,  $x=\sum_{g_i \in J_{H_{1}}} r_{i}g_{i} + \sum_{h_k \in {J_{H_{2}}}^{s-1}}t_{k}h_k, \textnormal{where} \ r_{i}, t_{k} \in S$.  Suppose $x \in J$, then $r_{i},t_{k} \in J$ for all $i$ and $k$, since $(E(H_{1}) \cup E(H_{2}))\cap E(H_{3}) = \emptyset$. Therefore, $x \in IJ$, but $IJ \subset J\cdot( x_{k_{0}},y_{k_{0}},J_{\mathcal{P}} )$, which is a contradiction. Thus $x\notin J$.  
  
  \vspace{2mm}
  
From Theorem \ref{Thm3.IJ} it follows that $\reg S/(I \cap J) = 2+ \sum_{j=1}^m{s_{(j)}}$. Then applying regularity lemma to the short exact sequence (\ref{eq:IJ}) yields that for all $s>2$, $$\reg{{S}/{(I+J)}} = 2(s-1)+ \sum_{j=1}^m{s_{(j)}},$$ i.e. 
$$\reg{\frac{S}{(d_1,\ldots,d_{i})+J_{G}^{s}: d_{i+1}}} = 2(s-1)+ \sum_{j=1}^m{s_{(j)}}.$$ 
Thus from Lemma \ref{Lemma1.Reg}(a) we obtain that  $\reg{{S}/{((d_1,\ldots,d_{i})+J_{G}^{s}})} = 2s + \sum_{j=1}^m{s_{(j)}}$, for $s >2$. And for $s=2$ the statement follows from Lemma \ref{Lemma4.T1s=2}. This completes the proof. 
\end{proof}

\begin{remark}
Note that the path graphs and star graphs are $\mathcal{T}_{m}$ graphs, for suitable $s(i)$ and $m$. Using Theorem \ref{Thm4.T1}, we derive the following results.
\begin{enumerate}[a)]
    \item \cite[Theorem 3.5]{JAR20} Let $G = K_{1,n}$ be a star graph for $n \geq 3$. Then, $\reg S/J_{G}^s = 2s$ for all $s \geq 1$. 
    \item \cite[Observation 3.2]{JAR20} Let $G = P_{n}$ be a path graph for $n \geq 2$. Then, $\reg S/J_{G}^s = 2s +n-3$ for all $s \geq  1$.
\end{enumerate}
\end{remark}

\begin{remark} Note that $T$-type graphs discussed in Shen and Zhu \cite{SZ}  are special case of $\mathcal{T}_{m}$ graphs. Hence we obtain  \cite[Theorem 4.7]{SZ} as a particular case of Theorem \ref{Thm4.T1}. 
\end{remark}
Now, we focus on the regularity of powers of the binomial edge ideal of $H_{m}$ graphs.
\begin{theorem} \label{Thm4.T2}
Let $G$ be a $\mathcal{H}_m$ graph on $[n+1]$. Let $d_1,\ldots,d_n$ be a sequence of edge binomials of $G$ such that $d_1,\ldots,d_n$ forms a $d$-sequence. Set $d_0 =  0 \in S$. Then, for any $i=0,1,\ldots,n-1$, and for any $m \geq 3$, and for all $s\geq 1$, we have
$$\reg \frac{S}{(d_1,\ldots,d_i)+J_{G}^s} = 2s + \sum_{j=1}^{m+1}{s_{(j)}} + 1,$$
where $s_{(j)}$ as defined in Notation \ref{Def:Hm}. In particular, $\reg {S}/{J_{G}^s} = 2s + \sum_{j=1}^{m+1}{s_{(j)}} + 1$, for all $s\geq 1$.
\end{theorem}
\begin{proof} The proof is by induction on $s$. For $s=1$, the statement follows from  Lemma \ref{Thm3.T2}. We may assume the statement for $s\geq 2$, and the assertion holds for $s-1$ (take $d_1,\ldots,d_n$ to be in the same order as in Theorem \ref{Thm2.T2}).  Now, the proof for the statement $s$ is by descending induction on $i$. For $i=n-1$ the statement holds by Lemma \ref{Lemma4.T2dn}. Next, assume that the assertion holds for $i+1$. Consider the short exact sequence (\ref{eq:T1}). By induction hypothesis on $i$ it follows that $\reg {S}/{((d_1,\ldots,d_{i+1})+J_{G}^{s})} = 2s + \sum_{j=1}^{m+1}{s_{(j)}}+1$. By the similar argument as in Theorem \ref{Thm4.T1} we get $((d_1,\ldots,d_{i}):d_{i+1})+J_{G}^{s-1} =(d_1,\ldots,d_{i})+J_{G}^{s-1}+J_{K_{n}}$, for some $n < m$. 
Set $I= (d_1,\ldots,d_{i})+J_{G}^{s-1}$ and $J=J_{K_{n}}$. Consider the two cases: 

\vspace{2mm}

\textbf{Case 1.}
 If  $((d_1,\ldots,d_{i}):d_{i+1})+J_{G}^{s-1} = I$, then from induction hypothesis on $s$ it follows that $\reg {S}/{I} = 2(s-1) + \sum_{j=1}^{m+1}{s_{(j)}}+1$. Applying Lemma \ref{Lemma1.Reg}(a)  to the short exact sequence (\ref{eq:E1}) yields that $$\reg \frac{S}{(d_1,\ldots,d_{i})+J_{G}^{s}}= 2s + \sum_{j=1}^{m+1}{s_{(j)}}+1.$$ 

\vspace{2mm}

\textbf{Case 2.}
If $((d_1,\ldots,d_{i}):d_{i+1})+J_{G}^{s-1} =I + J$. Consider the  short exact sequence (\ref{eq:IJ}).
By induction hypothesis that $\reg {S}/{I} = 2(s-1) + \sum_{j=1}^{m+1}{s_{(j)}}+1$. Since $J$ is a complete graph, from Remark \ref{rem1.Reg}(a) it follows that $\reg {S}/{J} = 1$.  \\
\textbf{Claim.} $I \cap J = J\cdot( x_{k_{0}},y_{k_{0}},J_{\mathcal{P}} )$, where $\mathcal{P}$ denotes an induced subgraph of $G$ on $V(G) \setminus k_0$ and $k_0$ denote the center of $G$. 
The proof is similar to the proof of Claim \ref{Claim1} in Theorem \ref{Thm4.T1}.
Now, From Theorem \ref{Thm3.IJ} we obtain that $\reg S/(I\cap J)= 2+ \sum_{j=1}^{m+1}{s_{(j)}}$. Then, from the short exact sequence (\ref{eq:IJ}) it follows that $\reg{{S}/{(I+J)}} = 2(s-1)+ \sum_{j=1}^{m+1}{s_{(j)}}+1$, for all $s\geq 2$. Thus, applying Lemma \ref{Lemma1.Reg}(a) to the  short exact sequence (\ref{eq:E1}) yields  $$\reg{\frac{S}{(d_1,\ldots,d_{i})+J_{G}^{s}}} = 2s + \sum_{j=1}^{m+1}{s_{(j)}}+1,  \text{ for $s \geq 2$. } $$
\end{proof}

To complete the proof of the Theorem \ref{Thm4.T2}, we need to prove the following lemma.

\begin{lemma} \label{Lemma4.T2dn}
Under the assumption in Theorem \ref{Thm4.T2}, and for all $s \geq 2$, we have
$$\reg \frac{S}{(d_1,\ldots,d_{n-1})+ d_{n}^{s}} = 2s + \sum_{j=1}^{m+1}{s_{(j)}}+1.$$
\end{lemma}
\begin{proof} 
From Lemma \ref{Thm3.T2} it follows that $\reg {S}/{J_{G}} = 2 + \sum_{j=1}^{m+1}s_{(j)}+1$. By using Theorem \ref{Thm2.T2} and Remark \ref{Rem1.GUG} one can obtain that  $\reg {S}/{(d_1,\ldots,d_{n-1})} = 2 + \sum_{j=1}^{m+1}s_{(j)}$. The rest of the proof is similar to the proof of Lemma \ref{Lemma4.T1p2}.
\end{proof}

\begin{remark}
Note that $H$-type graphs discussed in \cite{SZ} is a particular case of graphs we considered in Theorem \ref{Thm4.T2}. Hence we obtain  \cite[Theorem 4.7]{SZ} as one specific case of Theorem \ref{Thm4.T2}. 
\end{remark}

\subsection*{Conclusion} The proof of Theorem~\ref{P1.2} follows from Theorems \ref{Thm4.T1} and \ref{Thm4.T2}.

\vspace{2mm}

We conclude the article with the following question:

\begin{question}
Classify all finite simple graphs such that their edge binomials form a $d$-sequence.
\end{question}

\end{document}